\documentclass{amsart}
\usepackage{amssymb, latexsym}

\DeclareMathOperator{\Alt}{Alt}
\DeclareMathOperator{\Bil}{Bil}
\DeclareMathOperator{\Quad}{Quad}

\DeclareMathOperator{\M}{\mathcal{M}}

\DeclareMathOperator{\tr}{tr}

\theoremstyle{plain}
\newtheorem{theorem}{Theorem}
\newtheorem{corollary}{Corollary}
\newtheorem{lemma}{Lemma}

\theoremstyle{definition}

\begin{document}
\title[Real mutually unbiased bases and groups of odd order]
{Real mutually unbiased bases and \\ 
representations of groups of odd order by\\
real scaled Hadamard matrices of 2-power size}

\author[R. Gow]{Rod Gow}
\address{School of Mathematics and Statistics\\
University College Dublin\\
 Ireland}
\email{rod.gow@ucd.ie}


\begin{abstract} 
We prove two results relating real mutually unbiased bases
and representations of finite groups of odd order, as follows. Let $q$ be a power of 2 and $r$ a positive integer. Then we can find a $q^{2r}\times q^{2r}$ real orthogonal matrix $D$, say, of multiplicative order $q^{2r-1}+1$,  whose 
$q^{2r-1}
+1$ powers $D$, \dots, $D^{q^{2r-1}+1}=I$ define $q^{2r-1}+1$ mutually unbiased bases in $\mathbb{R}^{q^{2r}}$. Thus the scaled matrices $q^rD$, \dots, $q^rD^{q^{2r-1}}$ are $q^{2r-1}$ different Hadamard matrices. When we take $q=2$, we achieve
 the maximum number of real mutually unbiased bases in dimension $2^{2r}$ using the elements of a cyclic group. 
 
 Our second result is this. Let $G$ be an arbitrary finite group of odd order $2k+1$, where $k\geq 3$. Then $G$ has a real representation $R$, say, of degree 
 $2^{2^{k-1}}$ such that the elements $R(\sigma)$, $\sigma\in G$, define $|G|$ mutually unbiased bases in $\mathbb{R}^{d}$, where $d= 2^{2^{k-1}}$. In addition,
 a group of order 5 defines five real mutually unbiased bases in $\mathbb{R}^{16}$ and  a group of order 3 defines three real mutually unbiased bases in $\mathbb{R}^{4}$. Thus, an
 arbitrary group of odd order has a faithful representation by real scaled Hadamard matrices of 2-power size.
\end{abstract}
\maketitle

\section{Introduction}

\noindent Let $d$ be a positive integer and let $\mathbb{R}^d$ denote the Hilbert space of dimension $d$ over the field $\mathbb{R}$ of real numbers. Let $g:\mathbb{R}^d\times \mathbb{R}^d\to \mathbb{R}$ denote a positive definite
symmetric bilinear form.  Let $B_1=\{\,u_1, \ldots, u_d\,\}$ and $B_2=\{\,v_1, \ldots, v_d\,\}$ be  bases of $\mathbb{R}^d$ that are orthonormal with respect to $g$. We say that the bases $B_1$ and $B_2$ are \emph{mutually unbiased} if 
\[
g(u_i,v_j)^2=\frac{1}{d}
\]
for $1\le i,j\le n$. We may identify the orthonormal bases $B_1$ and $B_2$ with 
$d\times d$ real orthogonal matrices $O_1$ and $O_2$, say. Then the condition that the bases are mutually unbiased is equivalent to saying that each entry of $O_1O_2^{-1}$ equals $\pm 1/\sqrt{d}$. 

It is known that the maximum number of mutually unbiased bases in $\mathbb{R}^d$ is $1+\frac{d}{2}$, but a combination of theory and
experimental evidence suggests that this upper bound is rarely met. We note, however, that the upper bound is met
when $d=2^{2r}$, where $r$ is a positive integer, \cite{CS}. The introduction to the paper \cite{BSTW} provides useful information 
on the problem of determining the maximum number of real mutually unbiased bases. 

When investigating several mutually unbiased bases, we may assume that the basis $B_1$ is the standard basis of $\mathbb{R}^d$, corresponding to the identity matrix $I_d$. Then if the basis $B_i$ corresponds to the orthogonal matrix $O_i$, $B_i$ is mutually unbiased with respect to $B_1$ if and only if $O_i$ is an Hadamard matrix scaled by $1/\sqrt{d}$. Similarly, $B_i$ and $B_j$ are mutually unbiased if and only if $O_iO_j^{-1}$ 
is also an Hadamard matrix scaled by $1/\sqrt{d}$. This fact explains why if $d\neq 2$, and $4$ does not divide $d$, there do not exist
two mutually unbiased bases in $\mathbb{R}^d$. 

The purpose of this paper is twofold. First, to show that if $q$ is a power of 2 and $r$ is a positive integer, we can find a $q^{2r}\times q^{2r}$ real orthogonal matrix $D$, say, of multiplicative order $q^{2r-1}+1$,  whose 
$q^{2r-1}
+1$ powers $D$, \dots, $D^{q^{2r-1}+1}=I$ define $q^{2r-1}+1$ mutually unbiased bases in $\mathbb{R}^{q^{2r}}$. Thus the scaled matrices $q^rD$, \dots, $q^rD^{q^{2r-1}}$ are $q^{2r-1}$ different Hadamard matrices. When we take $q=2$, we achieve
 the maximum number of real mutually unbiased bases in dimension $2^{2r}$ using the elements of a cyclic group. 
 
 Our second objective is a proof of the following result. Let $G$ be an arbitrary finite group of odd order $2k+1$, where $k\geq 3$. Then $G$ has a real representation $R$, say, of degree 
 $2^{2^{k-1}}$ such that the elements $R(\sigma)$, $\sigma\in G$, define $|G|$ mutually unbiased bases in $\mathbb{R}^{d}$, where $d= 2^{2^{k-1}}$. Concerning the cases
 $k=1$ and $k=2$ not covered by this result, we remark that a group of order 3 has a four-dimensional real representation that defines three mutually unbiased bases in 
  $\mathbb{R}^4$ and group of order 5 has a 16-dimensional real representation that 
  defines five mutually unbiased bases in 
  $\mathbb{R}^{16}$. We can thus say that any group of odd order is (faithfully) represented
  by real scaled Hadamard matrices of 2-power size. This theorem, in its present form, is only of interest as an existence theorem, since the representing matrices are generally
  of extraordinarily large size. 
  
  We conclude the paper with a short investigation of the character of a representation 
  of a finite group by scaled Hadamard matrices.

\section{Orthogonal geometry over finite fields of 2-power order}

\noindent Let $K$ be a field of characteristic 2 and let $V$ be a vector space of dimension $n$ over $K$. Let $\Quad(V)$ denote the $K$-vector space of quadratic forms defined on $V$. This has dimension $n(n+1)/2$. Let $\Bil(V)$ denote the $K$-vector space of bilinear forms defined on $V\times V$. $\Bil(V)$ has dimension $n^2$. Let $\Alt(V)$ be the subspace of $\Bil(V)$ consisting of alternating bilinear forms. We have $\dim \Alt(V)=n(n-1)/2$.

There is a $K$-linear transformation, $\theta$, say, from $\Bil(V)$ to $\Quad(V)$ defined by
\[
\theta(f)(v)=f(v,v)
\]
for all $v\in V$. The polarization of $\theta(f)$ is the alternating bilinear form
$F$ defined by $F(u,v)=f(u,v)+f(v,u)$ for all $u$ and $v$ in $V$.

It is clear that $\ker \theta =\Alt(V)$ and thus the image of $\theta$ has dimension
$n(n+1)/2$ by the rank-nullity theorem. Since $\dim \Quad(V)=n(n+1)/2$, it follows that
$\theta$ is surjective. Thus given any $Q\in \Quad(V)$, there exists $f\in \Bil(V)$
with $\theta(f)=Q$. We note that any other bilinear form $g$ with $\theta(g)=Q$
is expressible as $g=f+h$, where $h\in\Alt(V)$.

Let $G$ be a finite group of $K$-linear automorphisms of $V$. There are well known $K$-linear actions of $G$ on $\Bil(V)$ and $\Quad(V)$, which we now describe. Let $\sigma$ be an element of $G$ and let $f$ and $Q$ be elements of $\Bil(V)$ and $\Quad(V)$, respectively.
We define $f^\sigma$ and $Q^\sigma$ by
\[
f^\sigma(u,v)=f(\sigma u, \sigma v), \quad Q^\sigma(v)=Q(\sigma v)
\]
for all $u$ and $v$ in $V$. It is elementary to check that $f^\sigma\in \Bil(V)$ and $Q^\sigma\in \Quad(V)$. Furthermore, we have $(f^\sigma)^\tau=f^{\sigma\tau}$ and
$(Q^\sigma)^\tau=Q^{\sigma\tau}$ for all $\tau\in G$. Thus we have representations of
$G$ defined on $\Bil(V)$ and $\Quad(V)$. 

Suppose now that $G$ fixes a non-zero element $Q$ of $\Quad(V)$. 
 We say that $G$ acts as isometries of $Q$ and each element of $G$ is an isometry of $Q$. Likewise, if for some element $f$ of
$\Bil(V)$, we have $f^\sigma=f$ for all $\sigma\in G$, we say that $G$ acts as isometries of $f$, and each element of $G$ is an isometry of $f$. 

The question we wish to address in this section is this. Let $Q$ be a $G$-fixed quadratic form. Is there then some $f\in\Bil(V)$ fixed by $G$ that satisfies $\theta(f)=Q$?
 As we shall see, the answer is no in general, but it is certainly yes if $G$ has odd order. 

We investigate the question as follows. Let $g$ be an element of $\Bil(V)$ with
$\theta(g)=Q$. Then since $Q^\sigma=Q$ for all $\sigma\in G$, we also have
$\theta(g^\sigma)=Q$. Thus, by our earlier discussion, $g^\sigma$ and $g$ differ by an element of $\Alt(V)$, $h_\sigma$, say. Consequently,
\[
g^\sigma=g+h_\sigma.
\]

Let $\tau$ be any element of $G$. We have then
\[
g^{\sigma\tau}=g+h_{\sigma\tau}.
\]
But we also have
\[
g^{\sigma\tau}=(g^\sigma)^\tau=(g+h_\sigma)^\tau=g^\tau+h_\sigma^\tau=g+h_\tau+
h_\sigma^\tau.
\]
We deduce that
\[
h_{\sigma\tau}=h_\tau+h_\sigma^\tau,
\]
an equation that we recognize as a one-cocycle relation. 

\begin{lemma} \label{condition_for_fixed_bilinear_form}

Let $Q$ be a $G$-fixed quadratic form and let $g\in \Bil(V)$ satisfy $\theta(g)=Q$. 
Given $\sigma\in G$, define $h_\sigma\in \Alt(V)$ by $h_\sigma=g+g^\sigma$. Then there exists a $G$-fixed bilinear form $f$ that satisfies $\theta(f)=Q$ if and only if there exists $b\in \Alt(V)$ with
\[
h_\tau=b+b^\tau
\]
for all $\tau$ in $G$.
\end{lemma}

\begin{proof}
Suppose that there exists a $G$-fixed bilinear form $f$ that satisfies $\theta(f)=Q$. Then taking into account the existence of $g$, we must have
$f=g+b$ for some $b\in \Alt(V)$. We apply $\tau\in G$ to obtain
\[
g+b=f=f^\tau=g^\tau+b^\tau=g+h_\tau+b^\tau.
\]
It follows that $h_\tau=b+b^\tau$, as required.

Conversely, suppose that there exists $b\in \Alt(V)$ with
$h_\tau=b+b^\tau$
for all $\tau$ in $G$. We set $f=g+b$, where $g$ is defined as above. Then $\theta(f)=Q$
and 
\[
f^\tau=g^\tau+b^\tau=g+h_\tau+b^\tau=g+b=f,
\]
since $h_\tau=b+b^\tau$. Thus $f$ is a $G$-fixed bilinear form with the desired property.
\end{proof}

\begin{theorem} \label{sufficient_condition_for_fixed_bilinear_form}
Suppose that $G$ is a finite group of isometries of the quadratic form $Q$ in $\Quad(V)$.
Then if $|G|$ is odd, there exists a $G$-fixed bilinear form $f$ with $\theta(f)=Q$.
\end{theorem}

\begin{proof}
 Let $g$ be a bilinear form that satisfies $\theta(g)=Q$. Then, as we have seen, we have
 $g^\sigma=g+h_\sigma$ and $h_{\sigma\tau}=h_\tau+h_\sigma^\tau$ for all $\sigma$ and $\tau$ in $G$. We now define $b\in\Alt(V)$ by setting
 \[
 b=\sum_{\rho\in G} h_\rho.
 \]
 Given $\tau\in G$, we obtain
 \[
 b^\tau=\sum_{\rho\in G} h_\rho^\tau=\sum_{\rho\in G}(h_{\rho\tau}+h_\tau),
 \]
 by the cocycle relation. But 
 \[
 \sum_{\rho\in G} h_{\rho\tau}=\sum_{\rho\in G} h_\rho,
 \]
 since $G$ is a group. Moreover, as $G$ has odd order and $K$ has characteristic two,
 \[
 \sum_{\rho\in G} h_\tau=h_\tau.
 \]
 We deduce that
 \[
 b^\tau=b+h_\tau
 \]
 for all $\tau\in G$. Lemma \ref{condition_for_fixed_bilinear_form} now guarantees the existence
of  a $G$-fixed bilinear form $f$ with $\theta(f)=Q$.
\end{proof}

Let us now show that Theorem \ref{sufficient_condition_for_fixed_bilinear_form}
does not hold for certain finite groups of even order. Suppose that $G$ acts absolutely
irreducibly on $V$. Then, the algebra of all $K$-linear transformations $C$ of $V$
that satisfy $C\sigma=\sigma C$ for all $\sigma\in G$ consists of scalar multiples of the identity. Given this observation, we have the following result.

\begin{theorem} \label{non_existence_of_fixed_bilinear_form}

Suppose that $G$ is a non-trivial finite group of linear transformations of $V$ that fixes the non-zero quadratic form $Q$. Suppose also that $G$ acts absolutely irreducibly on $V$. 
Then, provided $K$ is perfect, there is no $G$-fixed bilinear form $f$ with $\theta(f)=Q$.

\end{theorem}

\begin{proof}
Let $b$ be the polarization of $Q$. Since $G$ acts irreducibly on $V$, $b$ is either 0 or non-degenerate. Suppose if possible that $b=0$. Then we have $Q(u+v)=Q(u)+Q(v)$, for all
$u$ and $v$ in $V$. Moreover, $Q(\lambda u)=\lambda^2Q(u)$ for all 
$u\in V$ and $\lambda\in K$.
Since $K^2=K$ under the hypothesis that $K$ is perfect, and $Q$ is non-zero, the subset
of all $u\in V$ with $Q(u)=0$ is a $G$-invariant subspace of $V$ of codimension 1. Thus, since $G$ acts irreducibly, we must have $\dim V=1$. This in turn implies that $G$ is trivial, since it fixes a non-zero quadratic form defined on a one-dimensional vector space.
This is a contradiction. It follows that $b$ is non-degenerate.

Now, as $G$ acts absolutely irreducibly on $V$, Schur's Lemma implies that any
$G$-fixed bilinear form $F$, say, defined on $V\times V$ is a scalar multiple of $b$ and is thus alternating. This in turn implies that $\theta(F)=0$. We deduce that 
there is no $G$-fixed bilinear form $f$ with $\theta(f)=Q$. 
\end{proof}

\section{Construction of a group from a bilinear form}

\noindent Let $f$ be an element of $\Bil(V)$. We define a multiplication
on the set $V\times K$ 
by setting
\[
(u,\lambda)(v,\mu)=(u+v,\lambda+\mu +f(u,v))
\]
for all $u$ and $v$ in $V$, and all $\lambda$ and $\mu$ in $K$. It is straightforward 
to see that this multiplication is associative, since $f$ is bilinear. Indeed, 
$V\times K$ is group with respect to the multiplication, with identity element $(0,0)$. The inverse of
$(u,\lambda)$ is $(u,f(u,u)+\lambda)$. We let $E_f$ denote this group.

We now make the assumption that $n=\dim V$ is even, say $n=2m$, where $m$ is an integer
and that $\theta(f)$ is a non-degenerate quadratic form, $Q$, say. We shall also assume that
$K$ is the finite field $\mathbb{F}_q$, where $q$ is a power of 2. In this case, the
group $E_f$ has order $q^{2m+1}$ and
the alternating bilinear form $F$ defined by $F(u,v)=f(u,v)+f(v,u)$ is non-degenerate.
These assumptions will remain in place for the remainder of our exposition.

\begin{lemma} \label{centralizer_order}

Let $u$ be a non-zero element of $V$. Then the element $(u,\lambda)$ of $E_f$ has centralizer of order $q^{2m}$. 

\end{lemma}

\begin{proof}
It is clear from the definition of the multiplication in $E_f$ that $(v,\mu)$ commutes with
$(u,\lambda)$ if and only if $f(u,v)=f(v,u)$. Invoking the form $F$, $f(u,v)=f(v,u)$
is equivalent to $F(u,v)=0$. Since $F$ is non-degenerate, and $u$ is non-zero, the elements
$v$ that satisfy $F(u,v)=0$ form a subspace of codimension one in $V$. This proves what we want.
\end{proof}

Let $Z(E_f)$ denote the centre of $E_f$ and $E_f'$ the commutator subgroup. It follows from the previous lemma that $Z(E_f)$ consists of the elements $(0,\lambda)$,
where $\lambda$ runs over $\mathbb{F}_q$, and thus $Z(E_f)$ is elementary abelian of order
$q$. 

\begin{lemma} \label{commutator_subgroup}
$E_f'$ coincides with $Z(E_f)$ and $E_f/E_f'$ is elementary abelian of order $q^{2m}$.
\end{lemma}

\begin{proof}
We set $x=(u,\lambda)$ and $y=(v,\mu)$. Then we calculate that
\[
x^{-1}y^{-1}xy=(0, F(u,v)).
\]
Since $F$ is non-degenerate, if we fix a non-zero $u\in V$, the subset $F(u,v)$, as $v$ runs over $\mathbb{F}_q$, is equal to $\mathbb{F}_q$. Thus, $E_f'=Z(E_f)$. Furthermore, since
the square of any element of $E_f$ lies in $Z(E_f)=E_f'$, $E_f/E_f'$ has exponent 2.
\end{proof}

We proceed to determine the irreducible complex characters of $E_f$. 

\begin{lemma} \label{characters}
The group $E_f$ has exactly $q^{2m}$ irreducible characters of degree one and $q-1$ different irreducible characters of degree $q^m$. If $\chi$ is one of these irreducible characters of degree $q^{m}$, we have $\chi(x)=0$ if $x\not\in Z(E_f)$, and
$\chi(x)=\pm \chi(1)$ if $x\in Z(E_f)$. 
\end{lemma}

\begin{proof}
The number of irreducible characters of $E_f$ of degree one is $|E_f/E_f'|=q^{2m}$. Let 
$x$ be an element of $G$ that is not contained in $Z(E_f)$ and let $C(x)$ denote its centralizer in $E_f$. The second orthogonality relation for group characters implies that
\[
\sum |\psi(x)|^2=|C(x)|=q^{2m},
\]
where the sum extends over all irreducible characters $\psi$ of $E_f$. Now if $\psi$ has degree one, we have $|\psi(x)|^2=1$. Thus the contribution of the characters
of degree one to the sum above is $q^{2m}$. We deduce that $\psi(x)=0$ if $\psi(1)>1$.

Let $\chi$ be an irreducible character of $G$ of degree greater than one and let
$D$ be a complex representation of $G$ with character $\chi$. Schur's Lemma implies that if $x\in Z(E_f)$, then $D(x)$ is a scalar multiple of the identity. Furthermore, since $x^2=1$, this scalar is $\pm 1$. The first orthogonality relation for group characters implies that
\[
\sum_{x\in E_f} |\chi(x)|^2=|E_f|=q^{2m+1}.
\]

Since we know that $\chi(x)=0$ if $x\not\in Z(E_f)$ and $|\chi(x)|=\chi(1)$ if $x\in Z(E_f)$, we deduce that
\[
|Z(E_f)|\chi(1)^2=q^{2m+1}
\]
and hence $\chi(1)=q^m$. Thus $\chi(x)=\pm q^m$ if $x\in Z(E_f)$, as required. Finally, as the sum of the squares of the degrees of the irreducible characters of $E_f$ equals $|E_f|=q^{2m+1}$, we see that $E_f$ has exactly $q-1$ irreducible characters of degree $q^m$.
\end{proof}

We say that a vector $u$ in $V$ is singular with respect to $Q$ if $Q(u)=0$ and that
a subspace $U$ of $V$ is totally singular with respect to $Q$ if all elements of $U$ are singular. 
Suppose then that $U$ is totally singular with respect to $Q$. We set $\widehat{U}=(U,0)$.
It is easy to verify that $\widehat{U}$ is an elementary abelian subgroup
of $E_f$. 

Now, as we are working over a finite field, it is well known that there is a totally singular subspace of $V$ of dimension $m-1$ or $m$. We say that $Q$ has Witt index $m$
if there is a totally singular subspace of $V$ of dimension $m$. We shall assume for the rest of this exposition that $Q$ has Witt index $m$, and we let $U$ be a totally singular subspace of $V$ of dimension $m$. 

\begin{lemma} \label{induced_character}
Suppose that $Q$ has Witt index $m$ and let $U$ be a totally singular subspace of $V$ of dimension $m$. Let $A$ be the subgroup $\widehat{U}Z(E_f)$ of $E_f$. Then $A$ is elementary abelian of order $q^{m+1}$. Moreover, let $\chi$ be an irreducible character of $G$ of degree $q^m$. Then there exists a character $\epsilon$, say, of $A$ with $\epsilon(x)=\pm 1$ for all $x\in A$ and $\chi=\epsilon^{E_f}$. Thus, $\chi$ is the character of a representation of $E_f$ defined over the field of rational numbers.
\end{lemma}

\begin{proof}
It is straightforward to see that $A$ is elementary abelian of order $q^{m+1}$, since 
$\widehat{U}\cap Z(E_f)=1$ and $\widehat{U}$ and $Z(E_f)$ are both elementary abelian.

Let $\chi$ be an irreducible character of $G$ of degree $q^m$. Since $A$ is elementary  abelian, 
all its irreducible characters have degree 1 and take the values $\pm 1$. Thus the restriction of $\chi$ to 
$A$ contains some irreducible character $\epsilon$, say, with  $\epsilon(x)=\pm 1$ for all
$x\in A$. 

The Frobenius reciprocity theorem implies that $\chi$ is a constituent of $\epsilon^{E_f}$.
Since $\epsilon^{E_f}$ has degree $|E_f:A|=q^m$ and $\chi(1)=q^m$, a comparison of degrees shows that  $\chi=\epsilon^{E_f}$. Since $\epsilon$ is defined over the rational numbers,
the same is true of $\chi=\epsilon^{E_f}$.
\end{proof}

We wish now to introduce a group $G$ of $K$-linear automorphisms of $V$ that acts as isometries of $f$. We define an action of $G$ on $E_f$ in the following way. Given $\sigma$ in $G$ and
$(u,\lambda)$ in $E_f$, we set
\[
\sigma(u,\lambda)=(\sigma (u), \lambda).
\]
It is easy to check that $G$ acts as a group of automorphisms of $E_f$, since $G$ acts as isometries of $f$. We note that $G$ fixes $Z(E_f)$ elementwise.

Let $\chi$ be an irreducible complex character of $E_f$ of degree $q^m$. We have shown in Lemma
\ref{induced_character} that $\chi$ is the character of a representation, $X$, say, of $E_f$ defined on a vector space over the rational numbers. Thus, we may certainly take $X$ to be defined on the
real vector space $\mathbb{R}^{q^m}$. We may also assume that $X$ preserves a positive
definite symmetric bilinear form, $B$, say, defined on $\mathbb{R}^{q^m}\times \mathbb{R}^{q^m}$.

Let $\sigma$ be an element of $G$. We define the conjugate representation 
$X^{\sigma}$ by 
\[
X^\sigma(x)=X(\sigma(x))
\]
for all $x\in E_f$. Since $\sigma$ fixes each element of $Z(E_f)$, and 
the character of $X$ vanishes outside $Z(E_f)$, it is clear that $X^\sigma$ also has character $\chi$. 
It follows that 
$X$ and $X^\sigma$ are equivalent representations and thus there exists an automorphism 
$D(\sigma)$, say, of  $\mathbb{R}^{q^m}$ satisfying
\[
X^{\sigma}(x)= X(\sigma(x))=D(\sigma)X(x)D(\sigma)^{-1}
\]
for all $x\in E_f$. Thus, the relation $D(\sigma)X(x)=X(\sigma(x))D(\sigma)$ holds.

\begin{lemma} \label{existence_of_the_two_cocycle}
Let $\sigma$ and $\tau$ be elements of $G$. Then there is a non-zero real number $c(\sigma,\tau)$ that satisfies
\[
D(\sigma\tau)=c(\sigma,\tau)D(\sigma)D(\tau).
\]
\end{lemma}

\begin{proof}
We have
\[
X^{\sigma\tau}(x)=D(\sigma\tau)X(x)D(\sigma\tau)^{-1}
\]
for all $x$ in $E_f$. But
\begin{eqnarray*}
X^{\sigma\tau}(x)=(X^\sigma)^\tau(x)=X^\sigma(\tau(x))&=&D(\sigma)X(\tau(x))D(\sigma)^{-1}\\
&=&D(\sigma)D(\tau)X(x)D(\tau)^{-1}D(\sigma)^{-1}.
\end{eqnarray*}
We see thus that $D(\sigma\tau)^{-1}D(\sigma)D(\tau)$ commutes with all elements
$X(x)$. Since $X$ is an absolutely irreducible representation of $E_f$, Schur's Lemma implies that
$D(\sigma\tau)^{-1}D(\sigma)D(\tau)$ is a scalar multiple of the identity and we consequently have
\[
D(\sigma\tau)=c(\sigma,\tau)D(\sigma)D(\tau),
\]
where $c(\sigma,\tau)$ is a non-zero real number.
\end{proof}

The numbers $c(\sigma, \tau)$ satisfy a well known two-cocycle relation, as we now demonstrate. 

\begin{lemma} \label{cocycle_relation}
Let $\sigma$, $\tau$ and $\rho$ be elements of $G$. Then we have 
\[
c(\sigma\tau, \rho)c(\sigma,\tau)=c(\sigma,\tau\rho)c(\tau, \rho).
\]
\end{lemma}

\begin{proof}
We apply the formula in Lemma \ref{existence_of_the_two_cocycle} to obtain
\begin{eqnarray*}
D((\sigma\tau)\rho)&=&c(\sigma\tau, \rho)D(\sigma\tau)D(\rho)\\
&=&c(\sigma\tau, \rho)c(\sigma,\tau)D(\sigma)D(\tau)D(\rho).
\end{eqnarray*}
Similarly, 
\begin{eqnarray*}
D(\sigma(\tau\rho))&=&c(\sigma,\tau\rho)D(\sigma)D(\tau\rho)\\
&=&c(\sigma,\tau\rho)c(\tau, \rho)D(\sigma)D(\tau)D(\rho).
\end{eqnarray*}
But $(\sigma\tau)\rho=\sigma(\tau\rho)$ by the associative law applied in $G$. It follows that
\[
c(\sigma\tau, \rho)c(\sigma,\tau)=c(\sigma,\tau\rho)c(\tau, \rho).
\]
\end{proof}

We provide a self-contained proof of a standard property of two-cocycles.

\begin{lemma} \label{the_order_of_the_cocycle}
Given $\sigma\in G$, set
\[
d(\sigma)=\prod_{\alpha\in G}c(\sigma,\alpha).
\]
Then we have
\[
c(\sigma,\tau)^{|G|}=d(\sigma)d(\tau)d(\sigma\tau)^{-1}
\]
for all $\tau$ in $G$. 
\end{lemma}

\begin{proof}
We take the equation $c(\sigma\tau, \rho)c(\sigma,\tau)=c(\sigma,\tau\rho)c(\tau, \rho)$
and let $\rho$ run over $G$, with $\sigma$ and $\tau$ fixed. Multiplying left and right hand sides of the $|G|$ equations we have
\[
c(\sigma,\tau)^{|G|}\prod_{\rho}c(\sigma\tau, \rho)=\prod_\rho c(\sigma,\tau\rho)
\prod_\rho c(\tau, \rho).
\]
It is clear that
\[
\prod_{\rho}c(\sigma\tau, \rho)=d(\sigma\tau),\quad \prod_\rho c(\tau, \rho)=d(\tau).
\]
Furthermore, as $G$ is a group, $\tau\rho$ runs over $G$ when $\rho$ runs over $G$. Thus
\[
\prod_\rho c(\sigma,\tau\rho)=d(\sigma).
\]
We have thus shown that $c(\sigma,\tau)^{|G|}d(\sigma\tau)=d(\sigma)d(\tau)$,
as required.
\end{proof}

Recall that we have a  positive
definite symmetric bilinear form, $B$, say, defined on $\mathbb{R}^{q^m}\times \mathbb{R}^{q^m}$ such that $B(X(x)u,X(x)v)=B(u,v)$ for all $x\in E_f$ and all $u$, $v$ in 
$\mathbb{R}^{q^m}$. Given $\sigma$ in $G$, we define $B^\sigma$ by
\[
B^\sigma(u,v)=B(D(\sigma)u,D(\sigma)v)
\]
all $u$, $v$ in 
$\mathbb{R}^{q^m}$. It is clear that $B^\sigma$ is also  a  positive
definite symmetric bilinear form. 

\begin{lemma} \label{scaling_factor}
For each element $\sigma$ of $G$, there is a positive real number $e(\sigma)$ such that
$B^\sigma=e(\sigma)B$.
\end{lemma}

\begin{proof}
We show that $B^\sigma$ is $X$-invariant. Given $x\in E_f$ and $u$, $v$ in 
$\mathbb{R}^{q^m}$, we have
\[
B^\sigma(X(x)u,X(x)v)=B(D(\sigma)X(x)u,D(\sigma)X(x)v).
\]
But we know that $D(\sigma)X(x)=X(\sigma(x))D(\sigma)$ and thus
\begin{eqnarray*}
B^\sigma(X(x)u,X(x)v)&=&B(X(\sigma(x))D(\sigma)u,X(\sigma(x))D(\sigma)v)\\
&=&B(D(\sigma)u,D(\sigma)v)=B^\sigma(u,v).
\end{eqnarray*}
This proves that $B^\sigma$ is $X$-invariant. Now as $B$ is also $X$-invariant, and
$X$ is absolutely irreducible, Schur's Lemma implies that there is a real number
$e(\sigma)$ with $B^\sigma=e(\sigma)B$. As $B$ and $B^\sigma$ are both positive definite,
$e(\sigma)$ is positive.
\end{proof}

We proceed to relate the numbers $e(\sigma)$ to the numbers $c(\sigma,\tau)$. 

\begin{lemma} \label{square_relation}
For all $\sigma$ and $\tau$ in $G$, we have
\[
c(\sigma,\tau)^2e(\sigma)e(\tau)=e(\sigma\tau).
\]
\end{lemma}

\begin{proof}
Let $u$ and $v$ be elements of
$\mathbb{R}^{q^m}$. We have
\begin{eqnarray*}
B^{\sigma\tau}(u,v)&=&B(D(\sigma\tau)u,D(\sigma\tau)v)\\
&=&B(c(\sigma,\tau)D(\sigma)D(\tau)u,c(\sigma,\tau)D(\sigma)D(\tau)v)\\
&=&c(\sigma,\tau)^2e(\sigma)e(\tau)B(u,v).
\end{eqnarray*}
Thus $B^{\sigma\tau}=c(\sigma,\tau)^2e(\sigma)e(\tau)B$. But we also have
$B^{\sigma\tau}=e(\sigma\tau)B$ and we deduce that
\[
c(\sigma,\tau)^2e(\sigma)e(\tau)=e(\sigma\tau).
\]
\end{proof}

We now make the assumption that $G$ has odd order, with $|G|=2r+1$, where $r$ is a positive integer. 

\begin{lemma} \label{coboundary_relation}
Suppose that $|G|=2r+1$, where $r$ is a positive integer. For $\sigma\in G$, we set
\[
b(\sigma)=d(\sigma)e(\sigma)^r,
\]
where $d(\sigma)$ and $e(\sigma)$ are defined as in Lemmas \ref{the_order_of_the_cocycle} and \ref{scaling_factor}, respectively.  Then we have
\[
c(\sigma,\tau)=b(\sigma)b(\tau)b(\sigma\tau)^{-1}
\]
for all $\tau\in G$.
\end{lemma}

\begin{proof}
We have shown that 
\[
c(\sigma,\tau)^{2r+1}=d(\sigma)d(\tau)d(\sigma\tau)^{-1}
\]
in Lemma \ref{the_order_of_the_cocycle} and
\[
c(\sigma,\tau)^2=e(\sigma\tau)e(\sigma)^{-1}e(\tau)^{-1}
\]
in Lemma \ref{scaling_factor}. It follows that
\[
c(\sigma,\tau)^{2r}=e(\sigma\tau)^re(\sigma)^{-r}e(\tau)^{-r}
\]
and hence 
\[
c(\sigma,\tau)=d(\sigma)e(\sigma)^r d(\tau)e(\tau)^r d(\sigma\tau)^{-1}e(\sigma\tau)^{-r}.
\]
This proves what we claimed.
\end{proof}

We now use the scalars $b(\sigma)$ to define a representation of $G$.

\begin{lemma} \label{representation}
Suppose that $|G|$ is odd. 
Given $\sigma\in G$, define $R(\sigma)=b(\sigma)D(\sigma)$. Then we have 
\[
R(\sigma\tau)=R(\sigma)R(\tau)
\]
for all $\tau\in G$. Thus $R$ is a representation of $G$ defined on $\mathbb{R}^{q^m}$ that satisfies $R(\sigma)X(x)=X(\sigma(x))R(\sigma)$ for all $x\in E_f$.
\end{lemma}

\begin{proof}
We have
\begin{eqnarray*}
R(\sigma\tau)=b(\sigma\tau)D(\sigma\tau)&=&b(\sigma\tau)c(\sigma,\tau)D(\sigma)D(\tau)\\
&=&b(\sigma\tau)c(\sigma,\tau)b(\sigma)^{-1}b(\tau)^{-1}R(\sigma)R(\tau)\\
&=&R(\sigma)R(\tau),
\end{eqnarray*}
since $c(\sigma,\tau)=b(\sigma)b(\tau)b(\sigma\tau)^{-1}$. This proves that $R$ is a representation of $G$, and $R(\sigma)X(x)=X(\sigma(x))R(\sigma)$, since the same equation
holds when $D(\sigma)$ replaces $R(\sigma)$.
\end{proof}

Finally, we prove that the elements $R(\sigma)$ leave $B$ invariant. 

\begin{lemma}
Suppose that $|G|$ is odd. 
Let $u$ and $v$ be arbitrary elements of  
$\mathbb{R}^{q^m}$ and let $\sigma$ be an element of $G$. Then we have
\[
B(R(\sigma)u,R(\sigma)v)=B(u,v).
\]
Thus $R(G)$ acts as isometries of $B$. 
\end{lemma}

\begin{proof}
The proof of Lemma \ref{scaling_factor} implies that there is a non-zero real scalar
$\mu(\sigma)$, say, such that 
\[
B(R(\sigma)u,R(\sigma)v)=\mu(\sigma)B(u,v).
\]
Since $R$ is a representation, it is clear that $\mu$ is a homomorphism from
$G$ into the multiplicative group of non-zero real numbers. However, as we are assuming that
$|G|$ is odd, this homomorphism must be trivial and thus $B(R(\sigma)u,R(\sigma)v)=B(u,v)$, as required.
\end{proof}

\section{Partial orthogonal spreads, real mutually unbiased bases and Hadamard matrices}

\noindent We assume as before $V$ is a vector space of even dimension $2m$ over
$\mathbb{F}_q$, where $q$ is a power of 2. Let $Q$ be a non-degenerate quadratic form
of Witt index $m$ defined on $V$. 

We call a set of $m$-dimensional totally singular subspaces of $V$ a \emph{partial orthogonal spread} for $V$ and $Q$ if the subspaces intersect
trivially in pairs. Since $V$ contains precisely
\[
 (q^{m-1}+1)(q^m-1)
\]
non-zero singular vectors, it follows that a partial orthogonal spread contains at most $q^{m-1}+1$ subspaces. We say that a spread is complete if it contains exactly
$q^{m-1}+1$ subspaces. It can be proved that if $m$ is odd, a partial orthogonal spread contains at most two subspaces, whereas if $m$ is even, complete spreads exist.

We wish to consider a partial orthogonal spread constructed in the following way. Let
$G$ be a group of isometries of $Q$ of odd order. Suppose that $U$ is a totally singular 
subspace of $V$ of dimension $m$. Then it is clear that $\sigma(U)$ is also totally singular
of dimension $m$ for any $\sigma$ in $G$. We make the hypothesis that $\sigma(U)\cap \tau(U)=0$ if $\sigma$ and $\tau$ are different elements of $G$. Then the $|G|$ subspaces
$\sigma(U)$, $\sigma\in G$, form a partial orthogonal spread that is permuted transitively and regularly by $G$. 

While the requirements of the hypothesis may seem difficult to meet, there are two cases in which they are fulfilled. One occurs when $m$ is even. In this case, for any $q$ that is a power of 2, there is a cyclic subgroup of order $q^{m-1}+1$ that permutes transitively
and regularly the subspaces in a complete orthogonal spread. A second case is as follows.
Let $G$ be an arbitrary group of odd order $2k+1$, where $k\geq 3$. Then there is a vector space $V$ of dimension $2^k$ over $\mathbb{F}_2$ and a partial orthogonal spread of size $|G|$ in $V$ that is permuted transitively and regularly by $G$. 

We continue with the hypothesis that $G$ is a group of isometries of odd order of $Q$ that
transitively and regularly permutes the subspaces in a partial orthogonal spread of size
$|G|$. By Theorem \ref{sufficient_condition_for_fixed_bilinear_form}, there is a $G$-invariant bilinear form $f$, say, with $\theta(f)=Q$. $G$ then acts as a group of automorphisms of the group $E_f$ of order $q^{2m+1}$. 

Let $X$ be an irreducible real representation of $E_f$ of degree $q^m$ and let
$\chi$ be its character. We have a corresponding representation $R$ of $G$ that satisfies
\[
R(\sigma)X(x)=X(\sigma(x))R(\sigma)
\]
for all $x\in E_f$ and preserves
an $X$-invariant positive definite symmetric bilinear form, $B$. 

Let $U$ be an $m$-dimensional totally singular subspace of $V$ and let $\widehat{U}$ be
the corresponding elementary abelian subgroup of $E_f$ of order $q^m$.
 Since $\chi(x)=0$ for any non-identity element $x$ of $\widehat{U}$ and $\chi(1)=q^m=
|\widehat{U}|$, it follows that $X$ defines the regular representation of $\widehat{U}$ by restriction. Now the irreducible characters of $\widehat{U}$ have degree 1, since the group is abelian, and they take only the values $\pm 1$, since each element has order 1 or 2. The underlying
real vector space $\mathbb{R}^{q^m}$ then splits as the direct sum of $q^m$ one-dimensional subspaces which are each invariant
under the action of $X(\widehat{U})$. Each of these one-dimensional subspaces determines a different irreducible (linear) character of $\widehat{U}$.

As  the elements of $X(\widehat{U})$ preserve the  symmetric bilinear form $B$, the one-dimensional subspaces of $\mathbb{R}^{q^m}$ that
we have described above are mutually orthogonal, and we may choose an orthonormal basis $Z_1$, \dots, 
$Z_{q^m}$ of $\mathbb{R}^{q^m}$ such that
\[
 X(x)Z_i=\lambda_i (x)Z_i
\]
for all $x\in \widehat{U}$ and $1\leq i\leq q^m$. Here, $\lambda_i$ is an irreducible character of $\widehat{U}$, 
with $\lambda_i(x)=\pm 1$, and $\lambda_1$, \dots, $\lambda_{q^m}$ constitute all the irreducible characters of $\widehat{U}$.

We are assuming that we can choose $U$
so that $\sigma(U)$, $\sigma\in G$, is a partial orthogonal spread of size $|G|$.
Then the $|G|$
 subgroups $\sigma(\widehat{U})$, $\sigma\in G$, of $E_f$ 
 are all elementary abelian and they intersect trivially in pairs. 
It follows from the equation $R(\sigma)X(x)=X(\sigma(x))R(\sigma)$ for all $x\in \widehat{U}$ that
$R(\sigma)Z_1$, \dots, $R(\sigma)Z_{q^m}$ is an orthonormal basis of eigenvectors of $X(\sigma(\widehat{U}))$ for
all $\sigma\in G$. We thus obtain $|G|$ orthonormal bases of $\mathbb{R}^{q^m}$ from these subgroups.

The basic fact that we need to know
is that these orthonormal bases are mutually unbiased. A proof of this fact is presented in the next theorem,
in the context of properties of elementary abelian 2-subgroups of real orthogonal matrices. We follow the basic ideas laid out in the proof
of Theorem 3.2 of \cite{BBRV}.

\begin{theorem} \label{mutually_unbiased}
 
Let $d=2^\ell$, where $\ell$ is a positive integer. Let $M$ and $N$ be elementary abelian $2$-subgroups of $d\times d$ real orthogonal matrices, with
$|M|=|N|=d$. Suppose that for all pairs $(x,y)\neq (I,I)$ in $M\times N$, 
 $\tr(xy)=0$,
where $\tr$ denotes the trace of the matrix. Then any two orthonormal bases of $\mathbb{R}^d$ consisting of simultaneous eigenvectors of
$M$ and $N$, respectively, are mutually unbiased. 
\end{theorem}

\begin{proof}
 We may assume that we have an orthonormal basis 
$X_1$, \dots, $X_{d}$ of $\mathbb{R}^{d}$ such that
\[
 xX_i=\lambda_i (x)X_i
\]
for all $x\in M$ and $1\leq i\leq d$, where $\lambda_1$, \dots, $\lambda_d$ are the irreducible characters of $M$.
Likewise, we may assume that we have a second orthonormal basis 
$Y_1$, \dots, $Y_{d}$ such that
\[
 yY_i=\mu_i (y)Y_i
\]
for all $y\in N$ and $1\leq i\leq d$, where $\mu_1$, \dots, $\mu_d$ are the irreducible characters of $N$. 

Consider now, for a given element $x$ of $M$, the $d\times d$ symmetric matrix
\[
X= \sum_{i=1}^d \lambda_i(x) X_iX_i^T.
\]
Since the vectors $X_1$, \dots, $X_d$ are orthonormal, we have 
\[
 XX_j=\sum_{i=1}^d \lambda_i(x) (X_iX_i^T)X_j=\sum_{i=1}^d \lambda_i(x) X_i(X_i^TX_j)=\lambda_j(x)X_j
\]
for any $X_j$. Thus $X$ acts on the basis in an identical manner to $x$ and we deduce that $X=x$. In the same way, we have
\[
 y=\sum_{i=1}^d \mu_i(y) Y_iY_i^T
\]
for any $y$ in $N$. 

Taking $x$ and $y$ as above, we obtain
\[
 xy=\sum_{1\leq i,j\leq d} \lambda_i(x) \mu_j(y) X_iX_i^TY_jY_j^T,
\]
and then we may take traces to obtain
\[
 \tr(xy)=\sum_{1\leq i,j\leq d} \lambda_i(x) \mu_j(y) \tr(X_iX_i^TY_jY_j^T). 
\]
However, by simple properties of the trace function,
\[
 \tr(X_iX_i^TY_jY_j^T)=\tr(X_i^TY_jY_j^TX_i)=(X_i^TY_j)^2
\]
and thus we have
\[
 \tr(xy)=\sum_{1\leq i,j\leq d} \lambda_i(x) \mu_j(y) (X_i^TY_j)^2.
\]

Now by hypothesis, $\tr(xy)=0$  and hence
\[
 0=\sum_{1\leq i,j\leq d} \lambda_i(x) \mu_j(y) (X_i^TY_j)^2,
\]
unless $x=y=I$. When $x=y=I$, we have
\[
 d=\sum_{1\leq i,j\leq d} (X_i^TY_j)^2.
\]

As $(x,y)$ runs over the elements of the direct product $M\times N$, we obtain a system of $d^2$ linear
equations in the $d^2$ unknowns $(X_i^TY_j)^2$, $1\le i,j\leq d$. The coefficient matrix of the system is
the $d^2\times d^2$ matrix $C$, say, which is the character table of $M\times N$. Taking $(I,I)$ as the first
element of $M\times N$, all entries of the first row of $C$ equal 1, and $C$ satisfies $CC^T=d^2I$. 

Our system of equations takes the form 
\[
 \Psi=C\Phi,
\]
 where $\Phi$ is the column vector of size $d^2$, whose entries are the numbers $(X_i^TY_j)^2$, and
$\Psi$ has its first entry equal to $d$, and all other entries equal to 0. Then we have
\[
 d^2 \Phi=C^T\Psi,
\]
and, since all entries of the first column of $C^T$ are 1, we deduce that all entries of $\Phi$ equal
$1/d$. It follows that
\[
 (X_i^TY_j)^2=\frac{1}{d}
\]
for all $i$ and $j$, and this equality implies that the bases are mutually unbiased.
\end{proof}

\begin{corollary} \label{generation_of_bases}
Let $V$ be a vector space of even dimension $2m$ over $\mathbb{F}_q$, where $q$ is a power of $2$ and $m$ is even. Let $Q$ be a non-degenerate quadratic form of Witt index $m$ defined on $V$ and let $G$ be a group of isometries of $Q$ of odd order. 

Suppose that $G$ permutes transitively and regularly $|G|$ totally singular subspaces of dimension $m$ in a partial orthogonal spread with respect to $Q$.

Then $G$ has a faithful orthogonal representation, $R$, say, on $\mathbb{R}^{q^m}$ such  that with respect to an orthonormal basis of  $\mathbb{R}^{q^m}$, the
matrices $R(\sigma)$, $\sigma\in G$, define $|G|$ mutually unbiased
bases transitively permuted by $G$. 

Thus, $G$ is faithfully represented by real scaled $q^m\times q^m$ Hadamard matrices.
\end{corollary}

\section{Construction of real mutually unbiased bases by group elements}

\noindent We provide proofs in this section for the two main theorems announced in the introduction to this paper. We require two results describing group actions on complete/partial orthogonal spreads. 

\begin{theorem} \label{complete_orthogonal_spead}
Let $q$ be a power of $2$ and let $V$ be a vector space of dimension $2m$ over $\mathbb{F}_q$, where $m$ is even. Let $Q\in \Quad(V)$ be a non-degenerate quadratic form of Witt index $m$. Then there exists a cyclic group $G$, say, of isometries of $Q$ of order $q^{m-1}+1$
and a totally singular subspace $U$ of $V$ of dimension $m$ such that the subspaces
$\sigma(U)$, $\sigma\in G$, are a complete orthogonal spread in $V$ of $q^{m-1}+1$ $m$-dimensional subspaces.
\end{theorem}

This theorem is proved in \cite{KW} (see Theorem 6.5 and the proof of Theorem 3.3 (i) in \cite{KW}). We are grateful to Bill Kantor for pointing out to us the existence of this theorem. 

The following method of generating real mutually unbiased bases is a consequence
of Theorem \ref{complete_orthogonal_spead} and Corollary \ref{generation_of_bases}.

\begin{corollary} \label{generation_of_real_mutually_unbiased_bases}
Let $q$ be a power of $2$ and let
$r$ be a positive integer. Then there exists a $q^{2r}\times q^{2r}$ real orthogonal matrix $D$, say, of multiplicative order $q^{2r-1}+1$,  whose 
$q^{2r-1}
+1$ powers $D$, \dots, $D^{q^{2r-1}+1}=I$ define $q^{2r-1}+1$ mutually unbiased bases in $\mathbb{R}^{q^{2r}}$. Thus the scaled matrices $q^rD$, \dots, $q^rD^{q^{2r-1}}$ are $q^{2r-1}$ different real Hadamard matrices. 
\end{corollary}

As we mentioned in the introduction, we can thus achieve the maximum number $2^{2r-1}+1$ of real mutually unbiased bases in $\mathbb{R}^{2^{2r}}$ by using the elements of a cyclic subgroup of real orthogonal matrices. This is an analogue of a result for complex
mutually unbiased bases in  $\mathbb{C}^{2^{r}}$, where we use a cyclic group of unitary matrices of order $2^r+1$. See, for example, Theorem 1 of \cite{G}.

It may be of interest to know if we can achieve this maximum number of real mutually unbiased bases using non-cyclic finite groups of orthogonal matrices. At present, we know of
only one such non-cyclic generation occurrence, which we shall describe here. We have shown
in \cite{G2}, p.8, that an elementary abelian group of order 9 acts in a transitive manner on a complete orthogonal spread defined on a quadratic space of dimension 8 over
$\mathbb{F}_2$. Corollary \ref{generation_of_bases} implies the following generation property. 

\begin{corollary} \label{non_cyclic_generation_of_real_mutually_unbiased_bases}
There exists a finite subgroup of real orthogonal $16\times 16$ matrices isomorphic to an elementary abelian group of order $9$ whose elements determine nine mutually unbiased bases
in $\mathbb{R}^{16}$.
\end{corollary}

We observe that if $p$ is a prime divisor of $q^{2r-1}+1$, where $q$ is a power of 2, then
2 has even order modulo $p$. When we consider the quadratic character of $2$ modulo $p$, this implies that  $p\not\equiv 7\bmod 8$. Thus, for example, $7$ does not divide $q^{2r-1}+1$ for any choice of $r$ and 2-power $q$, and we cannot realize seven mutually unbiased in $\mathbb{R}^{q^{2r}}$ using a group of order 7 by this process (although, as we shall see next, it can be done by a different method in 
$\mathbb{R}^{16}$).

We turn to consideration of the second main theme of this paper.

\begin{corollary} \label{groups_of_odd_order}
 Let $G$ be an arbitrary finite group of odd order $2k+1$, where $k\geq 3$. Then $G$ has a real representation $R$, say, of degree 
 $2^{2^{k-1}}$ such that the elements $R(\sigma)$, $\sigma\in G$, define $|G|$ mutually unbiased bases in $\mathbb{R}^{d}$, where $d= 2^{2^{k-1}}$.
 \end{corollary}
 
 \begin{proof}
Corollary 1 of \cite{G2} shows that $G$ acts in a regular transitive manner on a partial orthogonal spread of size $|G|$ defined on a quadratic space of dimension $2^k$ over
$\mathbb{F}_2$. We deduce the required corollary from Corollary \ref{generation_of_bases}.
 \end{proof}
 
 We note that Corollary \ref{generation_of_real_mutually_unbiased_bases} implies that
 a group of order 5 defines five real mutually unbiased bases in $\mathbb{R}^{16}$ (take $q=4$, $r=1$) and  a group of order 3 defines three real mutually unbiased bases in $\mathbb{R}^{4}$ (take $q=2$, $r=1$). These are the two cases not addressed in Corollary \ref{groups_of_odd_order}. Thus,  every group of odd order is (faithfully) represented
  by real scaled Hadamard matrices of 2-power size.

\section{Character theory}

\noindent As we have shown that a group of odd order is represented by real Hadamard matrices, it is of interest to investigate the character of any such representation. Corollary \ref{odd_order_case} is our main finding.

\begin{lemma} \label{the_S_matrix}
Let $D$ be a real orthogonal matrix representation of the finite group $G$ and let 
$\theta_D$ be the character of $D$. Let $\chi$ be an irreducible complex character of $G$. Define the matrix $S_\chi$ by
\[
S_\chi=\frac{\chi(1)}{|G|}\sum_{g\in G} \chi(g) D(g).
\]
Then $S_\chi$ is a complex hermitian matrix that satisfies $S_\chi^2=S_\chi$ ($S_\chi$ is an idempotent). Furthermore, if $\tr(S_\chi)$ denotes the trace of $S_\chi$, we have
\[
\tr(S_\chi)=\chi(1)(\theta_D,\chi)
\]
and this trace also equals the rank of $S_\chi$.

\end{lemma}

\begin{proof}
We set $S=S_\chi$ for simplicity.
Since $D$ is orthogonal, $D(g)^T=D(g^{-1})$ for each element $g$ of $G$. Thus,
\[
S^T=\frac{\chi(1)}{|G|}\sum_{g\in G} \chi(g)D(g)^T=\frac{\chi(1)}{|G|}\sum_{g\in G}
\chi(g)D(g^{-1}).
\]
We take complex conjugates above to obtain
\[
\overline{S}^T=\frac{\chi(1)}{|G|}\sum_{g\in G}\overline{\chi(g)}D(g^{-1}),
\]
where we use the fact that the matrices $D(g)$ are real. However, as $\chi$ is a complex character, $\overline{\chi(g)}=\chi(g^{-1})$. Thus
\[
\overline{S}^T=\frac{\chi(1)}{|G|}\sum_{g\in G} \chi(g^{-1})D(g^{-1})=S
\]
and we see that $S$ is hermitian.

We now calculate that
\[
S^2=\frac{\chi(1)^2}{|G|^2}\sum_{g,h\in G}\chi(g)\chi(h)D(g)D(h)=
\frac{\chi(1)^2}{|G|^2}\sum_{g,h\in G}\chi(g)\chi(h)D(gh).
\]
Thus, given $x$ in $G$, the coefficient of $D(x)$ in $S^2$ is
\[
\frac{\chi(1)^2}{|G|^2}\sum_{g\in G}\chi(g)\chi(g^{-1}x).
\]

The generalized orthogonality relation (Theorem 2.13 of \cite{I}) shows that
\[
\frac{\chi(1)^2}{|G|^2}\sum_{g\in G} \chi(g)\chi(g^{-1}x)=\frac{\chi(1)}{|G|}\chi(x)
\]
and we deduce that $D(x)$ has the same coefficient in $S^2$ as it has in $S$. Therefore,
$S^2=S$, and we have proved that $S$ is an idempotent.

When we take the trace of $S$, we obtain
\[
\tr(S)=\frac{\chi(1)}{|G|}\sum_{g\in G} \chi(g)\theta_D(g)=\chi(1)(\theta_D,\chi),
\]
where we use the fact that $\theta_D$ is real-valued, since it is the character of a real representation. Finally, as $S$ is an idempotent, its eigenvalues are 0 and 1.
Thus, $\tr(S)$ is the multiplicity of 1 as an eigenvalue of $S$. But $S$ can be diagonalized, as it is an idempotent, and thus its rank is equal to the multiplicity of 1 as an eigenvalue.
\end{proof}

\begin{corollary} \label{chi_not_a_constituent_of_theta}

In the circumstances of Lemma \ref{the_S_matrix}, suppose that $\chi$ is not a constituent of $\theta_D$. Then $S_\chi=0$. Thus, if for each $g$ in $G$, $d_{ij}(g)$ is the $i,j$-entry of $D(g)$,
\[
\sum_{g\in G} d_{ij}(g)\chi(g)=0.
\]
\end{corollary}

\begin{proof}
The hypothesis implies that $(\theta_D, \chi)=0$ and hence $\tr S_\chi=0$ by Lemma \ref{the_S_matrix}. As we showed that $\tr(S_\chi)$ equals the rank of $S_\chi$, $S_\chi$ has rank 0
and hence $S_\chi=0$. Now the $i,j$-entry of $S_\chi$ is
\[
\sum_{g\in G} d_{ij}(g)\chi(g)
\]
and we deduce that this sum is zero for all $i$ and $j$.
\end{proof}

\begin{theorem} \label{non_zero_multiplicity}
Let $D$ be a representation of the finite group $G$ by scaled $m^2\times m^2$ Hadamard matrices, where $m$ is an even positive integer. Let $\chi$ be an irreducible complex character of $G$ of odd degree. Then if $\chi$ is not the trivial character of $G$, $\chi$ is a constituent of $\theta_D$. The same conclusion holds if 
$\chi$ is the trivial character, provided that $|G|$ is even.
\end{theorem}

\begin{proof}
Let $d_{ij}(g)$ be the $i,j$-entry of $D(g)$. If $g\neq 1$, we have
\[
d_{ij}(g)=\frac{\epsilon_{ij}(g)}{m},
\]
where $\epsilon_{ij}(g)=\pm 1$.

Suppose that $\chi$ is not a constituent of $\theta_D$. Then  choosing arbitrary indices $i$ and $j$ with $i\neq j$, Corollary \ref{chi_not_a_constituent_of_theta} implies that
\[
\sum_{g\neq 1} \epsilon_{ij}(g)\chi(g)=0.
\]
Let us assume first that $\chi$ is not the trivial character. The orthogonality relations for group characters show that
\[
\sum_{g}\chi(g)=0,
\]
the sum extending over all elements $g$ of $G$, and hence
\[
\sum_{g\neq 1}\chi(g)=-\chi(1).
\]

Since $\epsilon_{ij}(g)=\pm 1$ for all non-identity $g$, and the elements $\chi(g)$ are algebraic integers, we see that
\[
\sum_{g\neq 1} \epsilon_{ij}(g)\chi(g)=2\alpha+\sum_{g\neq 1}\chi(g),
\]
where $\alpha$ is an algebraic integer. It follows that
\[
2\alpha=\chi(1).
\]
This is impossible, as $\chi(1)/2$ is not an algebraic integer, since $\chi(1)$ is assumed to be odd. We deduce that $\chi$ is a constituent of $\theta_D$ in this case.

We now consider the case that $\chi$ is the trivial character and $|G|$ is even. If
the trivial character is not a constituent of $\theta_D$, we have
\[
\sum_{g\neq 1}\epsilon_{ij}(g)=0.
\]
This is also impossible, for the sum above is clearly odd, there being an odd number
of non-identity elements in $G$. Thus the trivial character is also a constituent of
$\theta_D$ when $|G|$ is even.
\end{proof}

\begin{corollary} \label{abelian_case}

Suppose that $G$ is a finite abelian group. Then in the circumstances of Theorem \ref{non_zero_multiplicity}, we have $|G|-1\leq m^2$ if $|G|$ is odd, and 
$|G|\leq m^2$ if $|G|$ is even.
\end{corollary}

\begin{corollary} \label{odd_order_case}

Suppose that $G$ is a group of odd order. Then in the circumstances of Theorem \ref{non_zero_multiplicity}, each non-trivial irreducible character of $G$ occurs as a constituent of $\theta_D$.
\end{corollary}

\begin{proof}
This follows from Theorem \ref{non_zero_multiplicity}, since the degree of each irreducible character divides $|G|$ and hence is an odd number.
\end{proof}

\end{document}